\documentclass{amsart}
\usepackage{amsmath,amsthm}%
\usepackage{amsfonts}%
\usepackage{amssymb}%
\usepackage[utf8]{inputenc}%
\usepackage[spanish,brazil,english]{babel}%
\usepackage{graphicx}%
\usepackage{latexsym,multicol,enumerate}%
\usepackage{color}%
\theoremstyle{plain}
\newtheorem{theorem}{Theorem}

\newtheorem{corollary}[theorem]{Corollary}

\newtheorem{definition}{Definition}

\newtheorem{proposition}[theorem]{Proposition}
\newtheorem{remark}[theorem]{Remark}

\numberwithin{equation}{section}
\begin{document}
\title[Bifurcation for Yamabe problem in manifolds with boundary]{Bifurcation results for the Yamabe problem on Riemannian manifolds with boundary}
\author{Elkin D. Cardenas }
\date{April 2016}

\begin{abstract}
	
We consider the product of a compact Riemannian manifold without boundary and null scalar curvature with a compact Riemannian manifold with boundary, null scalar curvature and constant mean curvature on the boundary. We use bifurcation theory to prove the existence of a infinite number of conformal classes with at least two non-homothetic Riemannian metrics  of null scalar curvature and constant mean curvature of the boundary on the product manifold.
\end{abstract}
\maketitle

\section{Introduction}
Let $(M,g)$ be a compact Riemannian manifold without boundary of dimension $m\geq3.$ It is well know that the metrics of constant scalar curvature in the conformal class  $[g]$ of $g$ can be characterized variationally as critical points of the Hilbert-Einstein functional on the conformal class $[g].$ The existence of metrics with constant scalar curvature in conformal class $[g]$ was established trough the combined works of Yamabe \cite{yam}, Trudinger \cite{trud}, Aubin \cite{aubin} and Schoen \cite{schoen1}.

Once  the existence of constant scalar curvature metrics conformal to $g$ has been established, it is natural do ask  about its uniqueness. For instance, if the Yamabe constant $Y[M,[g]]$ defined as the minimum of the normalized Hilbert-Einstein functional on $[g]$ is non-positive then  there exists only one of these metrics (up to homotheties). Also, if  $g$ is an Einstein metric which is not conformal to the round metric on the sphere $\mathbb{S}^n$, there exists one unique  metric of   unit volume and constant scalar curvature  in $[g]$ by a result of Obata \cite{Obata}. However, Pollack  proved in \cite{pollack} that if the Yamabe constant $Y[M,[g]]$ is positive then there are conformal classes with an arbitrarily large number of metrics of unit volume and constant scalar curvature sufficiently $C^{0}$- close to $[g]$.

Bifurcation techniques have been successfully used to study multiplicity of solutions for the Yamabe problem, both in the compact and the nonimpact case, see for instance \cite{bettiol-piccione-santoro}, \cite{bettiol-piccione} and \cite{bettiol-piccione2}. In particular, Lima et al. in \cite{sobbif-pic} using bifurcation techniques proved a result of multiplicity in an infinite number of conformal classes on product manifolds.

For  compact Riemannian manifolds with boundary and dimension $m\geq3$ similar problems to those mentioned above have been studied by different authors. For instance, Escobar proved in \cite{escs} that almost every  compact Riemannian manifold with boundary  
is conformally equivalent to a with null scalar curvature and constant mean curvature on the boundary. Uniqueness results were obtained by Escobar in \cite{escu}. 

Motivated by the results in \cite{sobbif-pic} in this paper using the bifurcation theory we study the multiplicity of metrics with null scalar curvature and constant mean curvature on conformal class of a product metric. More precisely, consider $(M_1,g^{(1)})$ a compact Riemannian manifold, without boundary, null scalar curvature and $(M_2,g^{(2)})$ a compact Riemannian manifold with boundary, null scalar curvature and constant mean curvature. Consider now the product manifold $M=M_1\times M_2$ and the family of Riemannian metrics $(g_t)_{t>0}$ with null scalar curvature and constant mean curvature on $M$ defined by $g_t=g^{(1)}\oplus tg^{(2)}$. The main result of the paper (Theorem\,\ref{thm}) states that if $(M_2,g^{(2)})$ has positive constant curvature mean on the boundary, then  there is a sequence $(t_n)$
strictly decreasing tending to 0 such that every element of this sequence is a bifurcation instant of the family $(g_t)_{t>0}$. For all other value of $t$ the family $(g_t)_{t>0}$ is locally rigid.  The fact that every $t_n$ is a bifurcation instant gives an entirely new sequence $(g_{i})$ of Riemannian metrics with null scalar curvature and constant mean curvature on boundary such that each $g_{i}\in [g_{t}]$ for some $t$ close to $t_n$ with $g_{i}$ non-isometric to $g_t$. This proves multiplicity for an infinite number of conformal classes.

The paper is organized as follows: in section 2 we recall some basic fact about  variational characterization of the Riemannian metrics with null scalar curvature and constant mean curvature on boundary conformally related to a metric $g$. In section 3 we give an abstract result of bifurcation and we examine the convergence of a bifurcating branch.
Finally, in section 4 we show our main result.

\section{Variational setting of the Yamabe problem}
Let $(M,g)$ be a compact Riemannian manifold with boundary and $m=\text{dim}\,M\geq3$. We will denote by $\mathbb{H}^1(M)$ the space of functions in $\mathbb{L}^2(M)$ with first weak derivatives  in $\mathbb{L}^2(M)$.  Endowed with the inner product 
\[\langle\varphi_1,\varphi_2\rangle=\int_{M}\Bigl(g(\nabla\varphi_1,\nabla\varphi_2)+\varphi_1\varphi_2\Bigr)\,\upsilon_g,\ \quad\varphi_1,\varphi_2\in\mathbb{H}^1(M),\]
$\mathbb{H}^1(M)$ is a Hilbert space. Let $E:\mathbb{H}^1(M)\longrightarrow\mathbb{R}$ be defined by
\[E(\varphi)=\int_{M}\Bigl(g(\nabla\varphi,\nabla\varphi)+\frac{m-2}{4(m-1)}R_g\varphi^2\Bigr)\upsilon_g+\frac{m-2}{2}\int_{\partial\,M}H_g\varphi^2\,\sigma_g.\]It is well known in literature that $\varphi\in\mathbb{H}^1(M)$ is a critical point of $E$ under the constraint $\mathcal{C}=\{\varphi\in\mathbb{H}^1(M)|\int_{\partial\, M}\varphi^{\frac{2(m-1)}{m-2}}\,\sigma_g=1\}$ if and only if the conformal metric $\tilde{g}=\varphi^{\frac{4}{m-2}}g$ has null scalar curvature  and constant mean curvature, see for instance \cite[Proposition~2,1]{escs} .

\subsection{The manifold of the normalized harmonic functions }

To study the bifurcation of metrics with null scalar curvature and constant mean curvature we will need to characterize these as critical points of $E$ on a special submanifold of $\mathbb{H}^1(M)$. For this purpos,  remember that a  function $\varphi\in\mathbb{H}^1(M)$ is harmonic if
	\[\int_{M}g(\nabla\varphi,\nabla\phi)\,\upsilon_g=0,\] 
for all $\phi\in\mathbb{H}^1_0(M),$ where $\mathbb{H}^1_0(M)$ is the kernel of trace operator, i.e.,
\[\mathbb{H}^1_0(M)=\{\phi\in\mathbb{H}^1(M)\Bigl|\,\phi|_{\partial M}=0\}.\]  

Let $\mathbb{H}^1_{\Delta}(M)$ denote the subspace of $\mathbb{H}^1(M)$ given by all harmonic functions. We know that $\mathbb{H}^1(M)=\mathbb{H}^1_0(M)\oplus\mathbb{H}^1_{\Delta}(M)$. Hence, $\mathbb{H}^1_{\Delta}(M)$ is a closed subspace  of Hilbert space $\mathbb{H}^1(M)$, and therefore  an embedded submanifold of it. By the Sobolev embedding,  there is a continuous inclusion $\mathbb{H}^1(M)\subset\mathbb{L}^{p}(\partial M)$ with  $p=\frac{2(m-1)}{m-2}$, see for instance \cite[Th. 2.21, pg. 19]{tedp}.
Let $\mathcal{H}_1(M,g)$ denote the subset of $\mathbb{H}^1_{\Delta}(M)$ consisting of those functions $\varphi$ such that $\int_{\partial M}\varphi^{\frac{2(m-1)}{m-2}}\,\sigma_g=1.$


\begin{proposition}\label{vsy}\hfill
\begin{enumerate}
\item $\mathcal{H}_1(M,g)$ is an embedded codimension 1 submanifold of   $\mathbb{H}^1_{\Delta}(M)$. In addition, for $\varphi\in\mathcal{H}_1(M,g)$ the tangent space $T_{\varphi}\mathcal{H}_1(M,g)$ of $\mathcal{H}_1(M,g)$ in $\varphi$ is given by
\begin{equation}\label{ts}
T_{\varphi}\mathcal{H}_1(M,g)=\Bigl\{\phi\in\mathbb{H}^1_{\Delta}(M)\Bigl|\ \int_{\partial M}\varphi^{\frac{m}{m-2}}\phi\,\sigma_g=0\Bigr\}.\end{equation}
\item Suppose that the metric $g$ has null scalar curvature. Then $\varphi\in\mathcal{H}_1(M,g)$ is a critical point of $E$ on $\mathcal{H}_1(M,g)$ if and only if $\varphi$ is a smooth function, positive and the conformal metric $\tilde{g}=\varphi^{\frac{4}{m-2}} g$ has null scalar curvature 
and $\partial M$ has constant $\tilde g$-mean curvature.\footnote{%
With a slight abuse of terminology, we will say that a metric $g$ on a manifold $M$ with boundary \emph{has constant mean curvature} meaning that the boundary $\partial M$ has constant mean curvature with respect to $g$.}\\

 Assume  that $\int_{\partial M}\sigma_g=1,$ and  $\varphi_0=\mathbf 1$\footnote{here $\mathbf 1$ denote the constant function 1 on $M$} 
 is a critical point of $E$ on $\mathcal{H}_1(M,g)$ (i.e., if $g$ has  null scalar curvature and constant mean curvature) then: 
 \item The second variation $d^2 E(\varphi_0)$  of $E$ at $\varphi_0$ is given by
\begin{equation}\label{sgv}
(\phi_1,\phi_2)\longrightarrow2\Bigl\{\int_{M}g(\nabla\phi_1,\nabla\phi_2)\,\upsilon_g-H_g\int_{\partial M}\phi_1\phi_2\,\sigma_g\Bigr\},\end{equation}
where $\phi_i\in \mathbb{H}^1_{\Delta}(M)$ is such that $\int_{\partial M}\phi_i\,\sigma_g=0,\ i=1,2.$
\item There exists a (unbounded) self-adjoint operator $J_g:\mathbb{L}_2(\partial M)\longrightarrow\mathbb{L}_2(\partial M)$ such that
\begin{equation}\label{svjo}
d^2 E(\varphi_0)(\phi_1,\phi_2)=2\int_{\partial M}J_g(\phi_1|_{\partial M})\phi_2\,\sigma_g,\end{equation}
where $\phi_i\in\mathbb{H}^1_{\Delta}(M)$ satisfies $\int_{\partial M}\phi_i\,\sigma_g=0,\ i=1,2.$
\end{enumerate}

\end{proposition}
\begin{proof}
Consider the function $\mathcal{V}_g:\mathbb{H}^1_{\Delta}(M)\longrightarrow\mathbb{R}$ defined by
\[\mathcal{V}_g(\varphi)=\int_{\partial M}\varphi^{\frac{2(m-1)}{m-2}} \,\sigma_g.\]
In order to prove (1), it is sufficient to show that $\mathcal{V}_g$  is a submersion at $\varphi,$ for all $\varphi\in\mathcal{H}_1(M,g)$. Note than $\mathcal{V}_g$ is smooth and its differential in $\varphi\in\mathcal{H}_1(M,g)$ is given by $d\mathcal{V}_g(\varphi)\phi=\frac{2(m-1)}{m-2}\int_{\partial M}\varphi^{\frac{m}{m-2}}\phi\,\sigma_g$, for all $\phi\in \mathbb{H}^1_{\Delta}(M)$. In particular, for $\phi=\varphi$ we get $d\mathcal{V}_g(\varphi)\phi=\frac{2(m-1)}{m-2}>0,$ therefore $d\mathcal{V}_g(\varphi)$ is surjective. Clearly the kernel of $d\mathcal{V}_g(\varphi)$ is complemented in $\mathbb{H}^1_{\Delta}(M)$, hence $\mathcal{V}_g$ is a submersion. Now for $\varphi\in\mathcal{H}_1(M,g)$, the tangent space $T_\varphi\mathcal{H}_1(M,g)$ of $\mathcal{H}_1(M,g)$ in $\varphi$ is the kernel of $d\mathcal{V}_g(\varphi)$ which establishes \eqref{ts}.

If the metric $g$ has null scalar curvature,  it is easy to see that $\varphi\in\mathcal{H}_1(M,g)$ is a critical point of $E$ on $\mathcal{H}_1(M,g)$ if and only if it is a critical point of $E$ under the constraint $\mathcal{C}=\{\varphi\in\mathbb{H}^1(M)|\int_{\partial\, M}\varphi^{\frac{2(m-1)}{m-2}}\,\sigma_g=1\}$. This proves (2). Assume $\int_{\partial M}\,\sigma_g=1$ and let $\varphi_0=\mathbf1$ be a critical point of $E$ on $\mathcal{H}_1(M,g)$. The formula \eqref{sgv}  is a straightforward computation  based on the method of Lagrange multipliers. 

Finally, given $\phi_1,\, \phi_2\in\mathbb{H}^1(M)$ we have

\[\int_{M}g(\nabla\phi_1,\nabla\phi_2)\,\upsilon_g=\int_{\partial M}\mathcal{N}_g(\phi_1|_{\partial M})\phi_2\,\sigma_g,\]
where $\mathcal{N}_g$ denotes the Direchlet-Neumann map, see for instance \cite{mazzeo-ar}. Hence from \eqref{sgv} we have
\[d^2E(\varphi_0)(\phi_1,\phi_2)=2\int_{\partial M}(\mathcal{N}_g(\phi_1|_{\partial M})-H_g\phi_1)\phi_2\sigma_g,\]
which proves (3) with $J_g(\phi_1)=\mathcal{N}_g(\phi_1|_{\partial M})-H_g\phi_1$.
 

%
\end{proof}

We will call $\mathcal{H}_1(M,g)$ the manifold of normalized harmonic functions.
\subsection{Jacobi Operator}

The unbounded linear operator $J_g:\mathbb{L}^2(\partial M)\longrightarrow\mathbb{L}^2(\partial M)$ defined by
\[J_g(\phi)=\mathcal{N}_g(\phi)-H_g\phi\]
is called the Jacobi operator. From the equation \eqref{svjo} follows that the dimension of $\text{Ker}J_g$ and the number (counted with multiplicity) of negative eigenvalues of $J_g$ are the nullity and the Morse index of $\varphi_0$ as a critical point of $d^2E$ in $\mathcal{H}_1(M,g)$. The spectrum of $J_g$ restricted to $\mathbf{L}_0(M)$, where
\[\mathbf{L}_0(M)=\{\phi\in\mathbb{L}^2(\partial M)|\ \phi=\varphi|_{\partial M},\ \varphi\in\mathbb{H}^1_{\Delta}(M)\},\]
is given by
\[-H_g<\rho_1(\mathcal{N}_g)-H_g\leq\rho_2(\mathcal{N}_g)-H_g\leq\rho_3(\mathcal{N}_g)-H_g\leq\cdots,\]
where the $\rho_{j}(\mathcal{N}_g)$ are repeated according to multiplicity. This proves the following
\begin{proposition}
Assume that $\varphi_0\equiv\mathbf1$ is a critical point of $E:\mathcal{H}_1(M,g)\longrightarrow\mathbb{R}$. Then:
\begin{enumerate}
	\item The Morse index $i(g)$ of $\varphi_0$ is given by
	\[i(g)=\max\{j|\rho_j(\mathcal{N}_g)<H_g\}.\]
	\item The nully of $\varphi_0,$ denoted with $\nu(g)$, is
	\[\nu(g)=\text{dim}\,\text{Ker}J_g.\]
\end{enumerate}	
	\end{proposition}
\begin{remark}
	Observe that $-H_g$ is not included in the spectrum of $J_g|_{\mathbf{L}_0(M)}$ since the only constant function on $\mathbf{L}_0(M)$ is the function null. We also recall the well known fact that eigenfunctions of $J_g$ are smooth and form an orthonormal basis of $\mathbf{L}_0(M)$.
\end{remark}
\subsection{Manifold of normalized Riemannian metrics on $M$}
For $k\ge2$, let $\mathbb{S}^{k}(M)$ denote the space of $\mathcal{C}^{k}-$sections of the vector bundle $T^*M\otimes T^*M$ of symmetric (0,2)-tensors of class $\mathcal{C}^{k}$ on $M$. $\mathbb{S}^{k}(M)$ has natural structure of Banach space. The set of all Riemannian metrics of class $\mathcal{C}^{k}$ on $M$ $\text{Met}\,^{k}(M)$ is an open subset of $\mathbb{S}^{k}(M)$ and hence an embedding  submanifold of it. For all $g\in\text{Met}\,^{k}(M)$ the tangent space $T_g\text{Met}\,^{k}(M)$ of $\text{Met}\,^{k}(M)$ in $g$ is given by $T_g\text{Met}\,^{k}(M)=\mathbb{S}^{k}(M)$ (see \cite{clarke} and the reference given there for more details on the manifold $\text{Met}\,^{k}(M)$).
 
 Let $\mathcal{V}:\text{Met}\,^{k}(M)\longrightarrow\mathbb{R}$ be defined by 
 \[\mathcal{V}(g)=\int_{\partial M}\sigma_{g}.\]
 It is easy to check that $\mathcal{V}$ is a submersion and therefore \[\mathcal{M}=\biggl\{g\in\text{Met}\,^{k}(M)\Bigl|\  \int_{\partial M}\sigma_{g}=1\biggr\},\]
 it  is an embedding submanifold of $\text{Met}\,^{k}(M)$. We will call  $\mathcal{M}$ the manifold of normalized Riemannian metrics on $M$.
 
 \section{Bifurcation of solutions and convergence of bifurcation branches}
 Let us consider a smooth curve  $[a,b]\ni t\longrightarrow g_t\in\mathcal{M}$ 
 with $\varphi_0=\mathbf1$ is a critical point of $E$ on $\mathcal{H}_1(M,g_t)$ for all $t\in[a,b]$
 (i.e., $g_t$ has null scalar curvature and constant mean curvature for all $t$).
 
 \begin{definition} 
 Let  $[a,b]\ni t\longrightarrow g_t\in\mathcal{M}$ be a curve as above. An element $t_*\in[a,b]$ said to be a bifurcation instant for the family  $(g_t)_{t\in[a,b]}$ if there exists a sequence $(t_n)\subset[a,b]$ and a sequence $(\varphi_n)\subset\mathbb{H}^1(M)$ such that
 \begin{enumerate}
 	\item[(\text{a})] $\varphi_n$ is a critical point of $E$ on $\mathcal{H}_1(M,g_{t_n})$ for all $n$;
 	\item[(\text{b})] $\varphi_n\neq\mathbf 1$ for all $n$;
 	\item[(\text{c})] $t_n\longrightarrow t_{*}$ as $n\longrightarrow+\infty$;
 	\item[(\text{d})]$\varphi_n\longrightarrow\mathbf1$ as $n\longrightarrow+\infty$ on $\mathbb{H}^1(M)$.
 	\end{enumerate}
 If $t_*\in[a,b]$ is not a bifurcation instant, then we say that the family $(g_t)_{t\in[a,b]}$ is locally rigid at $t_*$.
  \end{definition}

 Given $g\in\mathcal{M}$ with null scalar curvature  and constant mean curvature $H_g$, we say that $g$ is  nondegenerate if  either $H_g=0$ or  $H_g$ is not an eigenvalue of $\mathcal{N}_g$. The following bifurcation criterion is used to guarantee the conclusion of this work. 
 
 \begin{theorem}\label{scb}
 	Let $M$ be a compact manifold with boundary, with $m=\text{dim}\,M\geq3$. Let $[a,b]\ni t\longrightarrow g_t\in\mathcal{M}$ be a smooth curve of Riemannian metrics, with zero scalar curvature and constant mean curvature  for all 
 	$t\in[a,b]$. Given $t\in[a,b]$ let us denote by $\mathfrak{n}_t$ the number of  eigenvalues (counted with multiplicity) of $\mathcal{N}_{g_t}$ that are less than $H_{g_t}$. Suppose that:
 	\begin{enumerate}
 		\item The metrics $g_a$ and $g_b$ are nondegenerate,
 		\item $\mathfrak{n}_a\neq\mathfrak{n}_b$.
 	\end{enumerate}
 	Then, there exists a bifurcation instant $t_*\in\, ]a,b[$ for the family $(g_t)_{t\in[a,b]}$.
 \end{theorem}
  \begin{proof}
  	The result follows from the non-equivariant bifurcation theorem \cite[Theorem A.2]{sobbif-pic}
  	  	\end{proof}
    	
  \begin{remark}\label{ccbf}
  	It follows from implicit function theorem that if $t_*$ is a bifurcation instant for the family $(g_t)_{t\in[a,b]}$, then the mean curvature $H_{g_{t_*}}$ of the metric $g_{t_*}$ is a nonzero  eigenvalue of $\mathcal{N}_{g_{t_*}}$, see for instance \cite[Prop. 3.1]{sobbif-pic}. An instant $t\in[a,b]$ such that the mean curvature $H_{g_t}$ is a  nonzero eigenvalue  of $\mathcal{N}_{g_t}$ is called degeneracy instant of $(g_t)_{t\in[a,b]}$.
  \end{remark}

  \subsection[Convergence of bifurcation branches]{Convergence of bifurcation branches}	
  	
  Our purpose in this section is to settle the type of convergence of a bifurcating branch of solutions of the Yamabe problem in manifolds with boundary. 
  
  \begin{proposition}
  	Let $(g_t)_{t\in[a,b]}$ be a family of Riemannian metrics as above. If $t_*$ is a bifurcation instant for the family $(g_t)_{t\in[a,b]}$ then
  	\begin{enumerate}
  		\item $H_n\longrightarrow H_{t_{*}}$ as $n\longrightarrow+\infty$, where $H_n$ denotes the mean curvature of the metric $g_n=\varphi_{n}^{\frac{4}{m-2}}g_{t_{n}}$;
  		\item If $m\geq4$ then $\varphi_n\longrightarrow\mathbf1$ in $\mathbb{W}^{s,p}$ for all integer s and $p=\frac{2(m-2)}{m}$;
  		\item If $m\geq4$ then $\varphi_n\longrightarrow\mathbf1$ on $\mathcal{C}^{s}(\overline{M})$ for all integer $s$.
  	\end{enumerate}
  	\end{proposition}
  \begin{proof}
  	By Proposition \ref{vsy} for each $n\in\mathbb{N}$, the conformal metric $g_n=\varphi^{\frac{4}{m-2}}g_{t_n}$ has null scalar curvature  and constant mean curvature $H_{n}$. As $\varphi_n$ satisfies the equations
  	\begin{equation}\label{sc}
  		\begin{cases}
  		\Delta_{g_{t_n}}\varphi_n=0&\quad \text{in}\ M,\\[.2cm]
  		\dfrac{\partial\varphi_n}{\partial\eta^{t_n}}+\frac{m-2}{2}H_{g_{t_n}}\varphi_n=\frac{m-2}{2}H_n\varphi_n^{\frac{m}{m-2}}&\quad\text{on}\ \partial M,\\
  		\end{cases}
  		\end{equation}
partial integration in \eqref{sc} yields:
  	\begin{equation}\label{cmc}
  H_n=\frac{2}{m-2}\int_{M}g_{t_n}(\nabla^{t_n}\varphi_n,\nabla^{t_n}\varphi_n)\,\upsilon_{g_{t_n}}+H_{g_{t_n}}\int_{\partial M}\varphi_n^2\,\sigma_{g_{t_n}},
	 	\end{equation}
  	where $\nabla^{t_n}$ denotes the gradient operator calculated  respect to the metric $g_{t_n}$. Since $g_{t_n}\longrightarrow g_{t_*}$ in the $\mathcal{C}^{k}$-topology then $H_{g_{t_n}}\longrightarrow H_{g_{t_*}}$, $\upsilon_{g_{t_n}}=\psi_n\upsilon_{g_{t_*}}$ and $\sigma_{g_{t_n}}=\hat{\psi}_n\sigma_{g_{t_*}}$ where $\psi_n\longrightarrow\mathbf 1$ and
  	$\hat{\psi}_n\longrightarrow\mathbf1$ in the $\mathcal{C}^{k}$-topology. Moreover, since  $\varphi_n\longrightarrow\mathbf 1$ in $\mathbb{H}^1(M)$ then $$\int_{M}\sum\limits_{i,j=1}^{m}\frac{\partial\varphi_n}{\partial x^i}\frac{\partial\varphi_n}{\partial x^j}\,\upsilon_{g_{t_*}}\longrightarrow0\quad \text{as}\quad n\longrightarrow+\infty.$$
  	Hence 
  	\[\int_{M}g_{t_n}(\nabla^{t_n}\varphi_n,\nabla^{t_n}\varphi_n)\,\upsilon_{g_{t_n}}\longrightarrow0\quad \text{as}\quad n\longrightarrow+\infty.\]
  	On the other hand the continuity of the embedding $\mathbb{H}^1(M)$ into $\mathbb{L}^2(\partial M)$ yields
  	\[\int_{\partial M}\varphi_n^2\,\sigma_{g_{t_*}}\longrightarrow1\quad \text{as}\quad n\longrightarrow+\infty.\]
  	Thus
  		\[\int_{\partial M}\varphi_n^2\,\sigma_{g_{t_n}}\longrightarrow1\quad \text{as}\quad n\longrightarrow+\infty.\]
  	Therefore of \eqref{cmc} we get
  	\[H_n\longrightarrow H_{t_{*}}\quad \text{as}\quad n\longrightarrow+\infty.\]
  	
  	In order to prove (2) note that $v_n=\varphi_n-\mathbf1$ satisfy
  	\begin{equation}\label{sc2}
  	\begin{cases}
  	\Delta_{g_{t_n}}v_n=0 &\quad \text{in}\ M,\\[.2cm]
  	\dfrac{\partial v_n}{\partial\eta^{t_n}}+\frac{m-2}{2}H_{g_{t_n}}\varphi_n=\frac{m-2}{2}H_n\varphi_n^{\frac{m}{m-2}} &\quad\text{on}\ \partial M.\\
  	\end{cases}
  	\end{equation}
  	It is follows from standard elliptic estimates  (see for instance \cite[Th.\,15.2, pag. 704]{agmon}) that
  	\begin{multline}\label{es}
  	\|\varphi_n-\mathbf1\|_{\,\mathbb{W}^{s,p}(M)}\leq\\ C_n\biggl\{\Bigl\|-\frac{m-2}{2}H_{g_{t_n}}\varphi_n+\frac{m-2}{2}H_{g_n}\varphi_n^{\frac{m}{m-2}}\Bigr\|_{\, \mathbb{W}^{s-1-\frac{1}{p},p}(\partial M)}+\|\varphi_n-\mathbf1\|_{\,\mathbb{L}^p(M)}\biggr\},
  	\end{multline}
  	where $s$ is positive integer,  $1\leq p=\frac{2(m-2)}{m}<2$ and the constant $C_n$ depends only on $s$,  the manifold $M$, the $\mathbb{L}^{\infty}-\text{norm}$ of the coefficients of the elliptic operator $\Delta_{g_{t_n}}$, the ellipticity constant of $\Delta_{g_{t_n}}$, and the moduli of continuity of the coefficients of $\Delta_{g_{t_n}}$. As $g_{t_n}$ tend to $g_{t_*}$ in the $\mathcal{C}^{k}$-topology then the coefficients of $\Delta_{g_{t_n}}$ tend uniformly to the coefficients of the operator $\Delta_{g_{t_*}}$. Thus, in \eqref{es} we can choose a constant $C_n\equiv C$ that does not depend on $n$. Trace theorem says that the inclusion $\mathbb{W}^{s,p}(M)\hookrightarrow \mathbb{W}^{s-\frac{1}{p},p}(\partial M)$ is continuous hence of \eqref{es} we have
  	\begin{multline}\label{nes}
  	\|\varphi_n-\mathbf1\|_{\,\mathbb{W}^{s,p}(M)}\leq \overline{C}\Bigl\{\left\|-\frac{m-2}{2}H_{g_{t_n}}\varphi_n+\frac{m-2}{2}H_{g_n}\varphi_n^{\frac{m}{m-2}}\right\|_{\, \mathbb{W}^{s-1,p}( M)}\\+\|\varphi_n-\mathbf1\|_{\,\mathbb{L}^p(M)}\Bigr\},
  \end{multline}
  for some positive constant $\overline{C}$.
Since the inclusion $\mathbb{H}^1(M)\hookrightarrow \mathbb{L}^{p}(M)$ is continuous, then $\|\varphi_n-\mathbf1\|_{\,\mathbb{L}^p(M)}\longrightarrow0.$ Furthermore, the inclusion $\mathbb{H}^1(M)\hookrightarrow \mathbb{W}^{1,p}(M)$ is  continuous because $1\leq p<2.$ Hence for $s=2$ we get that $\varphi_n\longrightarrow\mathbf1$ in $\mathbb{W}^{s-1,p}$. We claim that $\varphi_n^{\frac{m}{m-2}}\longrightarrow\mathbf1$ in $\mathbb{W}^{1,p}(M)$. Indeed:
  	\begin{align*}\int_{M}(\varphi_n^{\frac{m}{m-2}}-1)^{\frac{2(m-2)}{m}}\,\upsilon_{g_{t_*}}&\leq \text{Vol}_{g_{t_*}}(M)^{\frac{2}{m}}\Bigl(\int_{M}(\varphi_n^{\frac{m}{m-2}}-1)^{2}\,\upsilon_{g_{t_*}}\Bigr)^{\frac{m-2}{m}}\\
  	&=\text{Vol}_{g_{t_*}}(M)^{\frac{2}{m}}\int_{M}\Bigl[\bigl(\varphi_n^{\frac{2m}{m-2}}-1\bigr)-2\bigl(\varphi_n^{\frac{m}{m-2}}-1\bigr)\Bigr]\,\upsilon_{g_{t_*}}.
  	  	\end{align*}
   The continuity of the inclusions $\mathbb{H}^1(M)\hookrightarrow\mathbb{L}^{\frac{2m}{m-2}}(M)$ and $\mathbb{H}^1(M)\hookrightarrow\mathbb{L}^{\frac{m}{m-2}}(M)$ imply that
  \[\int_{M}(\varphi_n^{\frac{m}{m-2}}-1)^{\frac{2(m-2)}{m}}\,\upsilon_{g_{t_*}}\longrightarrow0.\]
  Moreover for each $1\leq i\leq m$
  \begin{align*}
  \int_{M}(\partial_{i}\varphi_n^{\frac{m}{m-2}})^{\frac{2(m-2)}{m}}\,\upsilon_{g_{t_*}}&=\Bigl(\frac{m}{m-2}\Bigr)^{\frac{m-2}{m}}\int_{M}\varphi_n^{\frac{4}{m}}(\partial_i\varphi_n)^{\frac{2(m-2)}{m}}\,\upsilon_{g_{t_*}}\\
  &\leq\Bigl(\frac{m}{m-2}\Bigr)^{\frac{m-2}{m}}\Bigl(\int_{M}(\partial_i\varphi_n)^{2}\,\upsilon_{g_{t_*}}\Bigr)^{\frac{m-2}{m}}\Bigl(\int_{M}\varphi_n^{2}\,\upsilon_{g_{t_*}}\Bigr)^{\frac{2}{m}}.
  \end{align*}
  As the right side of the above inequality tends to zero then \[\int_{M}(\partial_{i}\varphi_n^{\frac{m}{m-2}})^{\frac{2(m-2)}{m}}\,\upsilon_{g_{t_*}}\longrightarrow0,\]this proves our claim. By \eqref{nes} we have that $\varphi_n\longrightarrow\mathbf1$ in $\mathbb{W}^{2,p}(M)$.
    Using induction on $s$ in equation \eqref{es} we obtain $\varphi_n\longrightarrow\mathbf1$ in $\mathbb{W}^{s,p}$ for all integer $s$. This proves (2).
  	 	
  	Finally, the proof of statement (3) is a direct consequence of the continuous inclusion $\mathbb{W}^{r+1,p}(M)\hookrightarrow \mathcal{C}^{s}(\overline{M})$, which  holds when $r>s-1+\frac{m}{p}$, see for instance \cite[Th.\,2.30, pag. 50]{aubinl}.
  \end{proof}

\begin{remark}
	It is important to stress that, for all $t\in(0,\infty)$, one has $\Delta_{tg}=\frac{1}{t}\Delta,\ H_{tg}=\frac{1}{\sqrt{t}}H_g$ and $\eta^{tg}=\frac{1}{\sqrt{t}}\eta^g$. This means that the spectrum of the \linebreak operator $\mathcal{N}_g-H_g$ is invariant by homethety of the metric. On the other hand, $\sigma_{tg}=t^{m-1/2}\sigma_g$. When needed, we will normalize metrics to have volume 1 on the boundary, without changing the spectral theory of the operator $\mathcal{N}_g-H_g$. 
\end{remark}

\section{Bifurcation of solutions for the Yamabe problem on product manifolds}

Let $(M_1,g^{(1)})$ be a compact Riemannian manifold without boundary and null scalar curvature, and let $(M_2,g^{(2)})$ be a compact Riemannian manifold with boundary, null scalar curvature  and constant mean curvature. Consider the product manifold, $M=M_1\times M_2,$ which boundary is given by $\partial M=M_1\times \partial M_2$. Let $m_1$ and $m_2$ be the dimensions of $M_1$ and $M_2$, respectively, and assume that $m=\text{dim}\,M=m_1+m_2\geq3.$ For each $t\in(0,+\infty)$, define $g_t=g^{(1)}\oplus tg^{(2)}$ a metric on the product manifold $M$. It is easily see that the Riemannian manifold $(M,g_t)$ have null scalar curvature  and constant mean curvature \[H_{g_t}=\frac{m_2-1}{(m-1)\sqrt{t}}H_{g^{(2)}}=\frac{\hat{H}_{g^{(2)}}}{\sqrt{t}},\quad\text{for all}\  t>0,\] 
where $\hat{H}_{g^{(2)}}=\frac{m_2-1}{(m-1)}H_{g^{(2)}}$. Let $\mathcal{H}_t$ denote the subspace of $\mathbb{L}^2(\partial M)$ consisting of  functions $\psi$   such that $\psi=\varphi|_{\partial M}, \varphi\in\mathbb{H}^1_{\Delta}(M)$ and $\int_{\partial M}\psi\,\sigma_{g_t}=0$ and let $\mathcal{E}_t$ denote the subspace of $\mathbb{L}^2(\partial M)$ of those maps $\phi$ such that $\int_{\partial M}\phi\,\sigma_{g_t}=0$. Let $\mathcal{J}_t:\mathcal{H}_t\longrightarrow\mathcal{E}_t$ be the Jacobi operator defined by 
\[\mathcal{J}_t(\psi)=\mathcal{N}_{g_t}(\psi)-H_{g_{t}}\psi.\]

Denote by $0=\rho^{(0)}<\rho^{(1)}<\rho^{(2)}<\ldots$ the sequence of all distinct eigenvalues of $\Delta_{g^{(1)}}$, with geometric multiplicity $\mu^{(i)},\ i\geq0,$  and by  $\rho^{(i)}_j(t)$ the $j-$ésimo eigenvalue of the problem
\begin{equation}\label{pavi}
\Delta_{g^{(2)}}\varphi+t\rho^{(i)}\varphi=0\ \text{in}\ M_2,\quad
\frac{\partial \varphi}{\partial\eta^{2}}=\rho\varphi\ \text{on}\ \partial M_2,
\end{equation}
with $t>0.$ We will denote by $\mu^{(i)}_j$ the geometric multiplicity of $\rho^{(i)}_j(t)$ and we emphasize that $\rho^{(i)}_j(t)$'s are not necessarily all distinct.  The spectrum $\Sigma(\mathcal{J}_t)$ of the Jacobi operator $\mathcal{J}_t$ is given by
\[\Sigma(\mathcal{J}_t)=\Bigl\{\frac{\rho^{(i)}_j(t)-\hat{H}_{g^{(2)}}}{\sqrt{t}}\Bigl|\ i+j>0,\quad i,j\in\mathbb{N}^*\Bigr\},\]
where $\mathbb{N}^*=\mathbb{N}\cup\{0\}$. We are interested in determining all values of $t$ for which $\rho^{(i)}_j(t)=\hat{H}_{g^{(2)}}$, for some $i,j\in\mathbb{N}^{*}$ with $i+j>0.$ In this case, the Jacobi operator $\mathcal{J}_t$ is degenerate.

\begin{remark}\hfill
	\begin{enumerate}
		\item[(\text{a})] For $i=0$ and $j\in\mathbb{N}^*$, $\rho^{(i)}_{j}(t)$ is an eigenvalue for $\mathcal{N}_{g^{(2)}}$. Hence if  $\hat{H}_{g^{(2)}}$ is an eigevalue for $\mathcal{N}_{g^{(2)}}$ then the Jacobi operator $\mathcal{J}_t$ is degenerate for all $t>0$.
		\item[(\text{b})] Given $\rho^{(i)}>0$ and $j\in\mathbb{N}^*$, using the variational characterization of $\rho^{(i)}_j(t)$ it is easily seen that 
		\[\rho^{(i)}_j(t_1)<\rho^{(i)}_j(t_2),\]
		for all $0<t_1<t_2$.
	\end{enumerate}
	\end{remark}

\begin{proposition}
	Assume that $\hat{H}_{g^{(2)}}$ is not an eigevalue for $\mathcal{N}_{g^{(2)}}$. For $i\in\mathbb{N}$,  let $\rho^{(i)}_0(t)$ denote the first eigenvalue of the problem \eqref{pavi}. Then,
	\begin{enumerate}
		\item The map $(0,\infty)\ni t\longrightarrow\rho^{(i)}_0(t)\in\mathbb{R}$ is continuous;
		\item $\rho^{(i)}_0(t)\longrightarrow0,$  as $t\longrightarrow0$;
		\item $\rho^{(i)}_0(t)\longrightarrow\infty$, as $t\longrightarrow\infty$;
		\item If $H_{g^{(2)}}>0,$ there exists an unique $t_{i}>0$ satisfying $\rho^{(i)}_0(t_i)=\hat{H}_{g^{(2)}}$. 
	\end{enumerate}  
\end{proposition}
\begin{proof}
Since the map 	$(0,\infty)\ni t\longrightarrow\rho^{(i)}_0(t)\in\mathbb{R}$ is strictly increasing, in order to establish $(1)$ we will show that given $t_0>0$ and any two sequences $(t^{(1)}_{n})_n$,  $(t^{(2)}_{n})_n$ such that $t^{(1)}_{n}\searrow t_0$ and $t^{(2)}_{n}\nearrow t_0$ we have
	\begin{equation}\label{cem}
	\inf_{n}\rho^{(i)}_0(t^{(1)}_{n})=\rho_0^{(i)}(t_0)=\sup_{n}\rho^{(i)}_0(t^{(2)}_{n}).\end{equation}
It is clear that $\rho^{(i)}_0(t_0)<\rho^{(i)}_0(t^{(1)}_{n})$ for all $n\in\mathbb{N},$ hence $\rho^{(i)}_0(t_0)\leq\inf\limits_{n}\rho^{(i)}_0(t^{(1)}_{n})$. On the other hand, by the variational characterization of $\rho^{(i)}_{0}(t_0)$, there exists  $\varphi\in\mathcal{C}^{\infty}(\overline{M_2})$ such that 
\[\rho^{(i)}_0(t_0)=E_{t_0}(\varphi):=\int_{M_2}\Bigl(g(\nabla\varphi,\nabla\varphi)+t_0\rho^{(i)}\varphi^2\Bigr)\,\upsilon_{g^{(2)}}.\]
Thus of the variational characterization of $\rho^{(i)}_0(t_n^{(1)})$ follows that
\[\rho^{(i)}_0(t_n^{(1)})\leq E_{t_n^{(1)}}(\varphi):=\int_{ M_2}\Bigl(g(\nabla\varphi,\nabla\varphi)+t_n^{(1)}\rho^{(i)}\varphi^2\Bigr)\,\upsilon_{g^{(2)}},\quad \text{for all}\ n\in\mathbb{N}.\]	
Therefore 
\[\inf_{n}\rho^{(i)}_0(t_n^{(1)})\leq\inf_{n}E_{t_n^{(1)}}(\varphi)=E_{t_0}(\varphi)=\rho^{(i)}_0(t_0),\]
which proves the left equality in \eqref{cem}. We will show now the equality of the right hand side in \eqref{cem}. Given the sequence $(t_n^{(2)})_n$ such that  $t_n^{(2)}\nearrow t_0$, it is evident that $\sup\limits_{n}\rho^{(i)}_0(t_n^{(2)})\leq\rho^{(i)}_0(t_0)$. Given $n\in\mathbb{N}$, let $\varphi_n\in\mathcal{C}^{\infty}(\overline{M_2})$ be such that
\[\rho^{(i)}_0(t_n^{(2)})=E_{t_n^{(2)}}(\varphi_n),\ \int_{\partial M_2} \varphi_n^{2}\,\sigma_{g^{(2)}}=1,\ \text{and}\ \int_{\partial M_2} \varphi_n\,\sigma_{g^{(2)}}=0, \]
this $\varphi_n$ exists because of the variational characterization of $\rho^{(i)}_0(t_n^{(2)})$. Note that $\{\varphi_n\}$ is a bounded subset in $\mathbb{H}^1(M_2)$. Since bounded sets are weakly compact in a Hilbert space, there exists a subsequence of  $\{\varphi_n\}$, denoted also by  $\{\varphi_n\}$, that converges weakly to a function $\psi\in\mathbb{H}^1(M_2)$. Then $\int_{\partial M_2} \psi\,\sigma_{g^{(2)}}=0$. Since  the  embeddings $\mathbb{H}^1(M_2)$ into $\mathbb{L}^2(\partial M_2)$ and into $\mathbb{L}^2( M_2)$  are compact, then the sequence $\{\varphi_n\}$ satisfies 
\[\int_{\partial M_2} \varphi_n^{2}\,\sigma_{g^{(2)}}\longrightarrow\int_{\partial M_2} \psi^{2}\,\sigma_{g^{(2)}}\ \text{and}\ \int_{ M_2} \varphi_n^{2}\sigma_{g^{(2)}} \longrightarrow\int_{ M_2} \psi^{2}\,\sigma_{g^{(2)}}.\]
Hence $\int_{\partial M_2} \psi^{2}\sigma_{g^{(2)}}=1.$ From the variational characterization of $\rho^{(i)}_0(t_0)$ we get \linebreak$\rho^{(i)}_0(t_0)\leq E_{t_0}(\psi)$. Since
\[\int_{ M_2} g(\nabla\psi,\nabla\psi)\,\sigma_{g^{(2)}}\leq\limsup_{n} \int_{ M_2} g(\nabla\varphi_n^{2},\nabla\varphi_n^{2})\,\sigma_{g^{(2)}},\]
then
\[E_{t_0}(\psi)\leq\limsup_{n}E_{t_n^{(2)}}(\varphi_n)=\limsup_{n}\rho^{(i)}_0(t_n^{(2)})=\sup_{n}\rho^{(i)}_0(t_n^{(2)}),\]
the inequality above establishes the equality of the right hand side in \eqref{cem}, which completes the proof of $(1)$. Now we will prove (2). To this end note that $0\leq\rho^{(i)}_0(t)$ for all $t>0$. Thus $0\leq\inf\limits_{t>0}\rho^{(i)}_0(t).$ On the other hand, it is known that there exists $C>0$ such that
\[\int_{ M_2}\varphi^2\,\upsilon_{g^{(2)}}\leq C\biggl\{\int_{ M_2}g(\nabla\varphi,\nabla\varphi)\,\upsilon_{g^{(2)}}+\int_{\partial M_2}\varphi^2\,\sigma_{g^{(2)}}\biggr\},\] 
for all $\varphi\in\mathbb{H}^1(M_2)$. Thus 
\[E_t(\varphi)\leq\int_{ M_2}g(\nabla\varphi,\nabla\varphi)\,\upsilon_{g^{(2)}}+t\rho^{(i)}C\biggl\{\int_{ M_2}g(\nabla\varphi,\nabla\varphi)\,\upsilon_{g^{(2)}}+\int_{\partial M_2}\varphi^2\,\sigma_{g^{(2)}}\biggr\}.\] 
Hence
\[\frac{E_t(\varphi)}{\int_{\partial M_2}\varphi^2\,\sigma_{g^{(2)}}}\leq(1+t\rho^{i}C)\frac{\int_{ M_2}g(\nabla\varphi,\nabla\varphi)\,\upsilon_{g^{(2)}}}{\int_{\partial M_2}\varphi^2\,\sigma_{g^{(2)}}}+t\rho^{i}C,\]
for all $\varphi\in\mathbb{H}^1(M_2)\backslash\{0\}$. Using the variational characterization of $\rho^{(i)}_0(t)$, it follows from the inequality above that 
\[0\leq\rho^{(i)}_0(t)\leq t\rho^{(i)}C,\quad \text{for all}\ t>0.\]
Therefore 
\[\rho^{(i)}_0(t)\longrightarrow0\quad  \text{as}\quad t\longrightarrow0.\]
 In order to prove (3) suppose the assertion is false. Let $\rho\in\mathbb{R}$ be such that $\lim\limits_{t\longrightarrow\infty}\rho^{(i)}_0(t)=\rho$. Clearly $\rho^{(i)}_0(t)\leq\rho$. Let $\varphi_t\in\mathcal{C}^{\infty}(\overline{M_2)}$ be such that
 \[\rho^{(i)}_0(t)=E_t(\varphi_t)\quad \text{and}\quad \int_{\partial M}\varphi_t^{2}\,\sigma_{g^{(2)}}=1.\]
   We claim that  $\int_{M_2}\varphi_{t}^2\upsilon_{g^{(2)}}\longrightarrow0$ at $t\longrightarrow\infty$. In fact, otherwise 
 \[\inf_{t>0}\int_{M_2}\varphi_{t}^2\upsilon_{g^{(2)}}>0,\]
 hence there exists $t>0$ such that
 \[\rho<t\inf_{t>0}\int_{M_2}\varphi_{t}^2\upsilon_{g^{(2)}}\leq\rho^{(i)}_0(t),\]
this contradicts the fact that $\rho^{(i)}_0(t)\leq\rho$. However $\{\varphi_t\}_{t>0}$   is a bounded set in $\mathbb{H}^1(M_2)$, thus there exists a sequence $(\varphi_{t_n})_n\subset\{\varphi_t\}_{t>0}$ such that
 \[\int_{\partial M_2}\varphi_{t_n}^2\upsilon_{g^{(2)}}\longrightarrow0\quad{as}\quad n\longrightarrow\infty,\]
which is impossible since $\int_{\partial M_2}\varphi_{t_n}^2\upsilon_{g^{(2)}}=1$ for all $n\in\mathbb{N}$. Therefore $\lim\limits_{t\longrightarrow\infty}\rho^{(i)}_0(t)=\infty$. Finally, $(4)$ is a direct consequence of (1)---(3).
	\end{proof}

\begin{corollary}\label{corb}
	Let $(M_1,g^{(1)})$ be a compact Riemannian manifold without boundary and scalar curvature null. Let $(M_2,g^{(2)})$ be a compact Riemannian manifold with boundary, scalar curvature null, and constant mean curvature $H_{g^{(2)}}$ . We will denote by $M$ the product manifold $M_1\times M_2$. Consider the family of Riemannian metrics $g_t=g^{(1)}\oplus tg^{(2)},\ t>0$, on $M$ and suppose that  $m=\text{dim}\,M\geq3$. Then:
	\begin{enumerate}
		\item[(\text{a})] If $H_{g^{(2)}}\leq0$, then the Jacobi operator is nondegenerate for all $t\in]0,\infty[$.
		In particular, the family  $(g_t)_{t>0}$ is locally rigid for all $t>0$. 
		\item [(\text{b})] If $H_{g^{(2)}}>0$ and  $\hat{H}_{g_{t}}$ is not a Steklov eigevalue for $\Delta_{g^{(2)}}$, there exists a strictly decreasing sequence tending  to zero consisting of degeneracy instants  of the family $(g_t)_{t>0}$. For all      other values of $t,$ the Jacobi operator $\mathcal{J}_t$ is nonsingular. 
	\end{enumerate}  
\end{corollary}
\begin{proof} 
	(a) follows easily from remark \ref{ccbf}.
For each $i\in\mathbb{N}$, let $t_i$ denote the unique positive real number such that $\rho^{(i)}_0(t_i)=\hat{H}_{g^{(2)}}$. For $0<i_1<i_2$ we obtain 
\[\hat{H}_{g^{(2)}}=\rho^{(i_1)}_0(t_{i_1})<\rho^{(i_2)}_0(t_{i_1})\leq\rho^{(i_2)}_0(t),\quad  t\geq t_{{i_1}}\]
Therefore $t_{i+1}<t_i$ for all $i\in\mathbb{N}.$ Hence the sequence $(t_i)_{i\in\mathbb{N}}$ is strictly decreasing. Clearly the Jacobi operator $\mathcal{J}_{t_i}$ is degenerate for all $i\in\mathbb{N}$. Let $\varphi_i\in\mathcal{C}^{\infty}(\overline{M_2})$ be such that 
\begin{equation}\label{inpo}
\rho^{(i)}_0(t_i)=E_{t_i}(\varphi_i)\quad \text{and}\quad \int_{\partial M_2}\varphi_i^2\,\sigma_{g^{(2)}}=1.
\end{equation}
As $\rho^{(i)}_0(t_i)=\hat{H}_{g^{(2)}}$ for all $i\in\mathbb{N}$, from first equality in \eqref{inpo} follows  that
\[\int_{M_2}g(\nabla\varphi_i,\nabla\varphi_i)\,\upsilon_{g^{(2)}}\leq \hat{H}_{g^{(2)}},\quad \text{for all}\ i\in\mathbb{N}.\]
Using the  inequality above and the second equality in \eqref{inpo}, we conclude that $\{\varphi_i\}_{i\in\mathbb{N}}$ is  bounded in $\mathbb{H}^1(M_2)$. We claim that $\inf\limits_{i}\int_{M_2}\varphi_i^{2}\upsilon_{g^{(2)}}>0.$ In fact, otherwise there exists a sequence $(\varphi_{t_{i_n}})_{n\in\mathbb{N}}\subset\{\varphi_i\}$ which converges weakly to 0 in $\mathbb{H}^1(M_2)$. Since the embedding of $\mathbb{H}_1(M_2)$ into $\mathbb{L}_2(\partial M)$ is compact then $\int_{\partial M_2}\varphi_{t_{i_n}}^2\,\sigma_{g^{(2)}}\longrightarrow0$ as $n\longrightarrow\infty.$
This contradicts the fact that $\int_{\partial M_2}\varphi_i^2\,\sigma_{g^{(2)}}=1$ for all $i\in\mathbb{N}.$
Thus
\[0<t_i=\frac{\hat{H}_{g^{(2)}}-\int_{M_2}g(\nabla\varphi_i,\nabla\varphi_i)\,\upsilon_{g^{(2)}}}{\rho^{i}\int_{M_2}\varphi_i^{2}\upsilon_{g^{(2)}}}\leq\frac{\hat{H}_{g_{(2)}}}{\rho^{i}\inf\limits_{i}\int_{M_2}\varphi_i^{2}\upsilon_{g^{(2)}}}\longrightarrow0,\]
as $i\longrightarrow\infty,$ Which	proves (b).
	
\end{proof}
We are ready for our main result.
\begin{theorem}\label{thm}
	Let $(M_1,g^{(1)})$ be a compact Riemannian manifold without boundary with scalar curvature zero and $(M_2,g^{(2)})$ a compact Riemannian manifold with boundary, having scalar curvature zero and  positive constant mean curvature $H_{g^{(2)}}$. Assume that  $\hat{H}_{g_{t}}$ is not a Steklov eigevalue for $\Delta_{g^{(2)}}$ and $\text{dim}\,(M_1)+\text{dim}\,(M_2)\geq3.$ For all $t\in(0,\infty)$, let $g_t=g^{(1)}\oplus tg^{(2)}$ be the metric on the product manifold with boundary, $M=M_1\times M_2$. Then there exists a sequence tending to zero consisting of bifurcation instants for the family $(g_t)_{t>0}$. For all other values of $t>0$ the  family $(g_t)_{t>0}$ is locally rigid.
\end{theorem}

\begin{proof}
	Corollary \ref{corb} establishes the existence of a strictly decreasing sequence $(\hat{t}_n)$ converging to 0 such that the Jacobi operator $\mathcal{J}_{\hat{t}_n}$ is singular for all $n\in\mathbb{N}$. Note that
	\[\hat{H}_{g^{(2)}}=\rho^{(1)}_0(\hat{t}_1)<\rho^{(i)}_0(t)\leq\rho^{(i)}_j(t),\quad i,j\in\mathbb{N},\ \text{and}\ t>\hat{t}_1.\]
		Therefore the Jacobi operator is nonsingular on $]\hat{t}_1,+\infty[.$ Our next claim is that there are at most a finite number of degeneracy instants for the family $(g_t)_{t>0}$ into $]\hat{t}_{i+1},\hat{t}_i[.$ Indeed, note that
	\[\hat{H}_{g^{(2)}}=\rho^{(i+1)}_0(\hat{t}_{i+1})<\rho^{(i+1)}_0(t)\leq\rho^{(r)}_j(t),\quad r\geq i+1,\ t>\hat{t}_{i+1},\ j\in\mathbb{N}^{*}.\]
	Hence, if $\rho^{(r)}_j(t)=\hat{H}_{g^{(2)}},\ t\in]\hat{t}_{i+1},\hat{t}_i[,$ then $1\leq r\leq i$. In addition,  for all $r$ there exists $J(r)\in\mathbb{N}$ such that 
	\[\hat{H}_{g^{(2)}}<\rho^{(r)}_j(\hat{t}_{i+1})<\rho^{(r)}_{j}(t),\quad j>J(r),\ t>\hat{t}_{i+1}.\]
Let $J_0=\max\limits_{1\leq r\leq i }\min J(r)$. Thus , if $\rho^{(r)}_j(t)=\hat{H}_{g^{(2)}}$ with $t\in\left]\hat{t}_{i+1},\hat{t}_i\right[$ then $1\leq r\leq i$ and $1\leq j\leq J_0$. Finally, for $i_1<i_2$ we have that $\rho^{(i_1)}_j(t)<\rho^{(i_2)}_j(t)$ for all $t$ and $j$. Therefore, for each pair $(r,j)$ there exists at most one $t=t(r,j)$ such that $\rho^{(r)}_j(t)=\hat{H}_{g^{(2)}}$, which proves our claim. Hence the set of all degeneracy instants for the family $(g_t)_{t>0}$ 
consists of a strictly decreasing sequence $(t_{i})_{i\in\mathbb{N}}$ converging to 0. For each $i\in\mathbb{N}$, there exists $\epsilon=\epsilon(i)>0$ such that 
\[[t_{i-\epsilon},t_{i+\epsilon}]\cap(t_{n})_{n\in\mathbb{N}}=\{t_i\}.\]
Moreover, for all $i$ and $j$ the map $t\longrightarrow\rho^{(i)}_j(t)$ is strictly increasing. Therefore $\mathfrak{n}_{t_{i}-\epsilon}\neq\mathfrak{n}_{t_{i}+\epsilon}$. Theorem \ref{scb} ensures that every degeneracy instant is a bifurcation instant.
\end{proof}



\end{document}